\documentclass[english]{extarticle}
\usepackage[T1]{fontenc}
\usepackage[latin9]{inputenc}
\usepackage{geometry}
\usepackage{float}
\usepackage{amsthm}
\usepackage{amsmath}
\usepackage{graphicx}
\usepackage{setspace}
\usepackage{url}
\usepackage{epstopdf}

\makeatletter
\numberwithin{equation}{section}
\numberwithin{figure}{section}
\theoremstyle{plain}
\newtheorem{thm}{\protect\theoremname}
  \theoremstyle{remark}
  \newtheorem{rem}[thm]{\protect\remarkname}
  \theoremstyle{plain}
  \newtheorem{cor}[thm]{\protect\corollaryname}

\usepackage{amsmath,amssymb,nopageno}
\usepackage[all]{xy}

\makeatother

\usepackage{babel}
  \providecommand{\corollaryname}{Corollary}
  \providecommand{\remarkname}{Remark}
\providecommand{\theoremname}{Theorem}

\begin{document}
\begin{onehalfspace}

\title{Solution of Sondow's problem: a synthetic proof of the tangency property
of the parbelos}
\end{onehalfspace}

\date{}

\begin{onehalfspace}

\author{Emmanuel Tsukerman}
\end{onehalfspace}
\maketitle
\begin{abstract}
In a recent paper titled \textit{The parbelos, a parabolic analog
of the arbelos}, Sondow asks for a synthetic proof to the tangency
property of the parbelos. In this paper, we resolve this question
by introducing a converse to Lambert's Theorem on the parabola. In the process, we prove some new properties of the parbelos.
\end{abstract}

\section{Introduction}

In a recent paper, Jonathan Sondow introduced the parbelos - a parabolic
analogue of the arbelos \cite{Parbelos}. One of the beautiful properties
of the parbelos is that the tangents at the cusps of the parbelos
form a rectangle, and that the diagonal of the rectangle opposite
the cusp is tangent to the upper parabola. Moreover, the tangency
point lies on the bisector of the angle at the cusp. Sondow asks for
a synthetic proof of these two properties of the tangent rectangle
of the parbelos, which he proves by analytic means. In this paper,
we present such a proof. We do so by introducing a converse to the
following Theorem of Lambert: the circumcircle of a triangle formed
by three tangent lines to the parabola passes through the focus of
the parabola. In the process of proving Sondow's tangent property,
we discover some new properties of the parbelos.

\begin{figure}[H]
\begin{raggedright}
\includegraphics[scale=0.32]{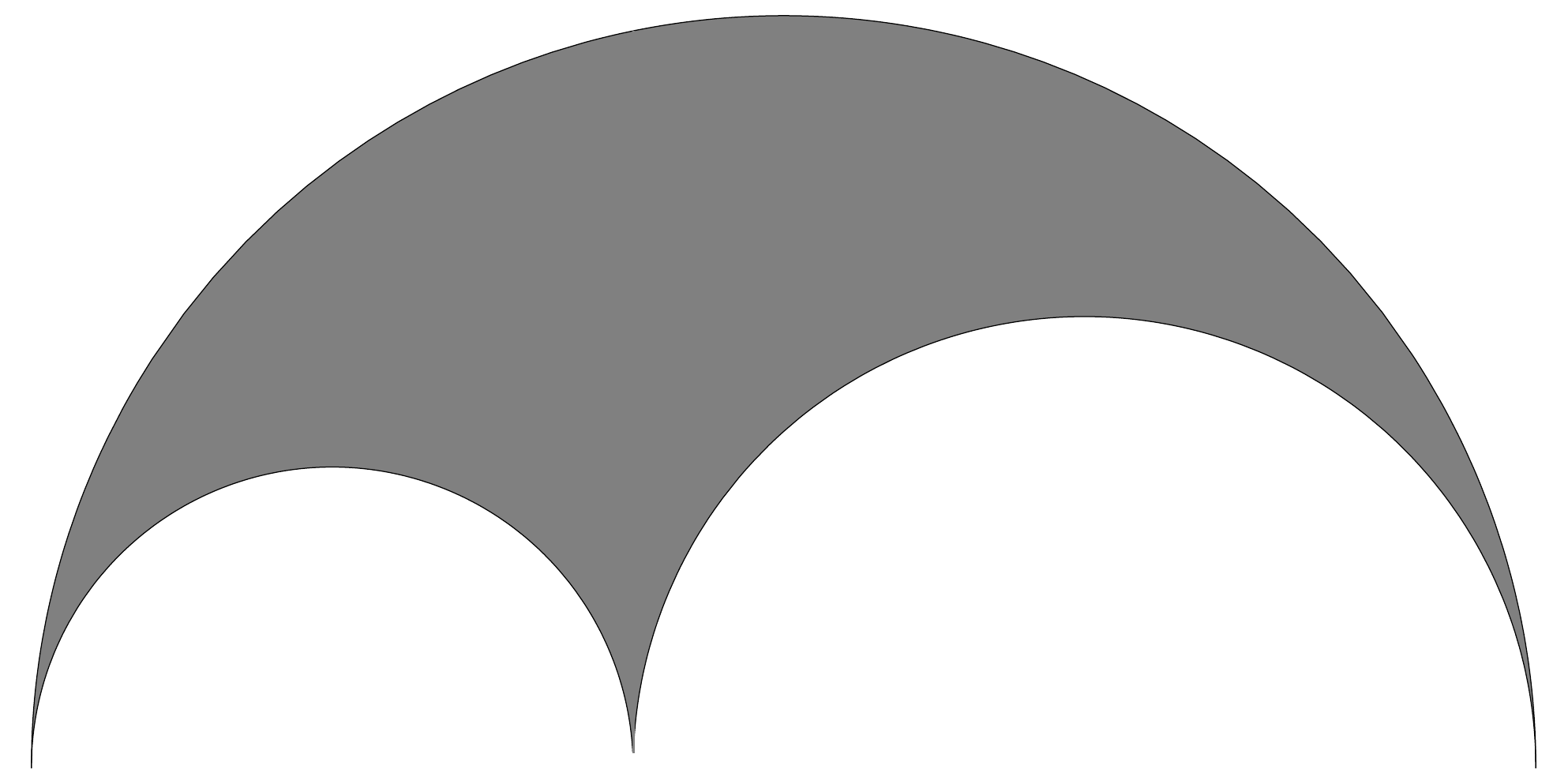}\includegraphics[scale=0.35]{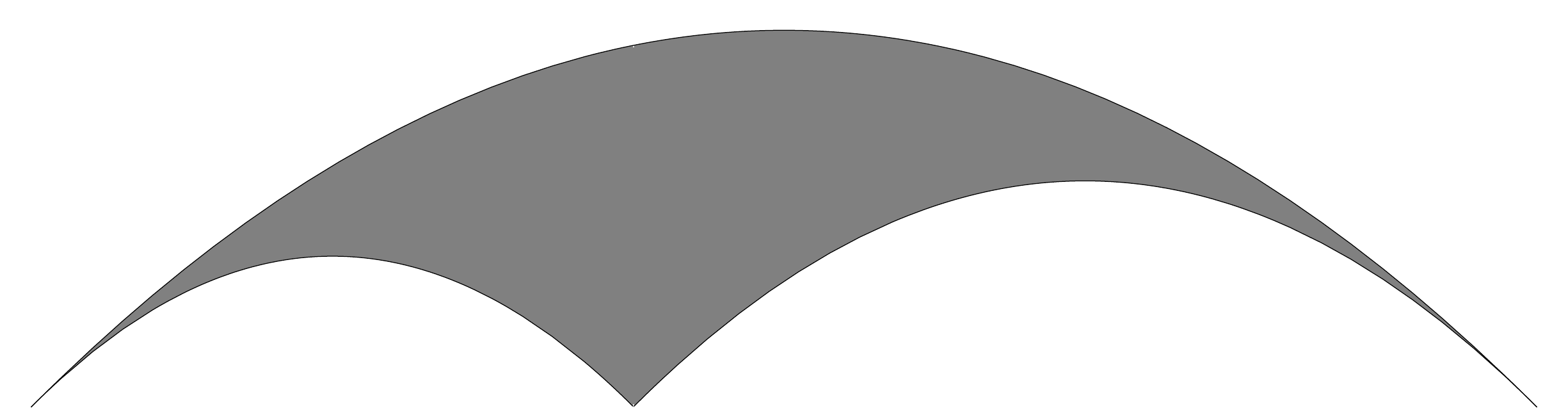}
\par\end{raggedright}

\caption{The arbelos and the parbelos.}
\end{figure}

\begin{figure}[H]
\begin{centering}
\includegraphics[scale=0.45]{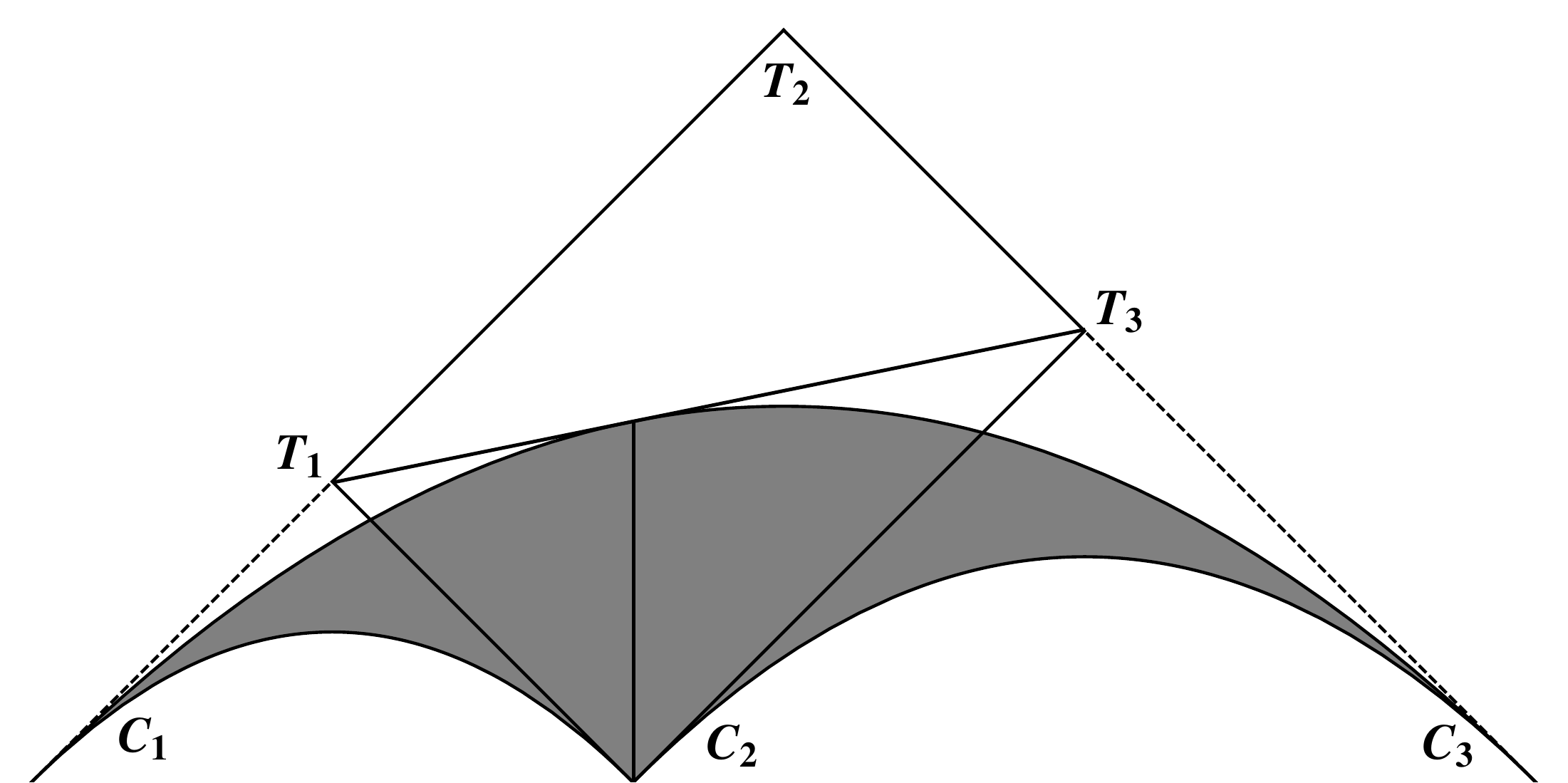}
\par\end{centering}

\caption{\label{fig:Sondow's-Tangency-Property:}Sondow's Tangency Property:
the diagonal $T_{1}T_{3}$ of the tangent rectangle $C_{2}T_{1}T_{2}T_{3}$
is tangent to the outer parabola. Moreover, the tangency point is
the intersection of the angle bisector of cusp $C_{2}$ with the outer
parabola.}
\end{figure}

\section{Preliminaries}

The classical Simson-Wallace Theorem is a useful tool in understanding
the parabola.
\begin{thm}
\label{thm:(Simson-Wallace-Theorem)-Given}(Simson-Wallace Theorem)
Given a triangle $\triangle ABC$ and a point $P$ in the plane, the
orthogonal projections of $P$ into the sides (also called pedal points)
of the triangle are collinear if and only if $P$ is on the circumcircle
of $\triangle ABC$ \cite{Geometry Revisited}. 
\end{thm}
In general, a pedal curve is defined as the locus of orthogonal projections
of a point into the tangents of the curve. In a sense discussed in
\cite{Simson}, the parabola may be viewed as a polygon with infinitely
many vertices which satisfies the following Simson-type property:
it is the unique curve such that its pedal curve with respect to a
point is a line. The point turns out to be the focus $F$ of the parabola
and the line is the supporting line at its vertex, which we will denote
by $\Lambda$. 
\begin{thm}
\label{thm:tangent to parabola iff relates to supporting line}A line
$l$ is tangent to the parabola if and only if the orthogonal projection
of the focus $F$ into $l$ lies on the supporting line $\Lambda$. \end{thm}
\begin{proof}
For a proof of the ``only if'' statement, we refer the reader to
\cite{Simson} and \cite{Hilbert}. Let $P$ be the orthogonal projection
of $F$ into $l$ and assume that $P\in\Lambda$. If $P$ is the vertex
of $G$, then clearly $l=\Lambda$ and we are done. So assume otherwise.
Since $\Lambda$ has no points inside of the parabola, there exists
a tangent line $\tilde{l}$ to $G$ not equal to $\Lambda$ which
passes through $P$. The ``only if'' part implies that the orthogonal
projection of $F$ into $\tilde{l}$ is on $\Lambda$, and is therefore
$P$. It follows that $l=\tilde{l}$.

\begin{figure}[H]
\noindent \begin{centering}
\includegraphics[scale=0.25]{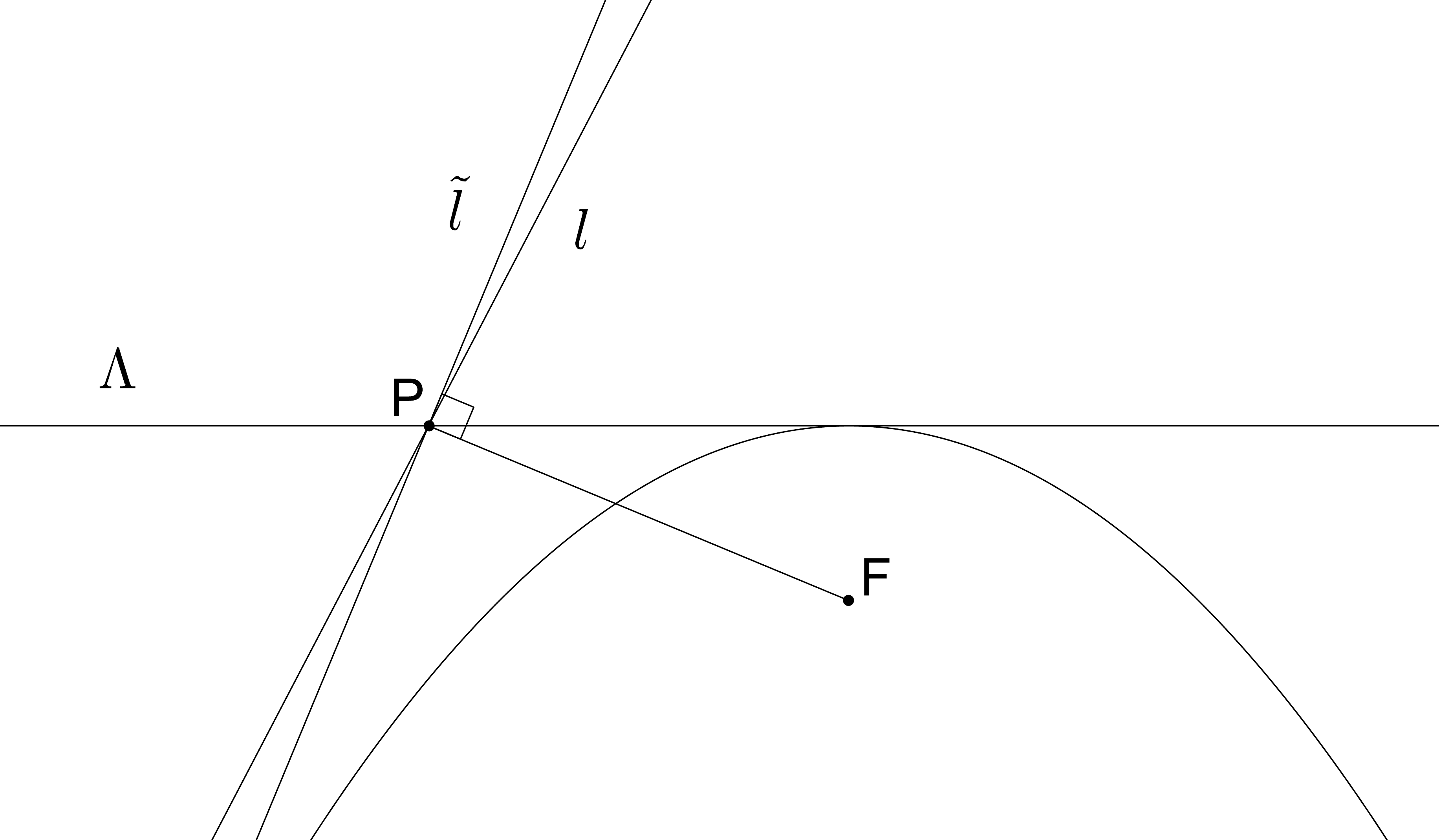}
\par\end{centering}

\caption{Proof of Theorem \ref{thm:tangent to parabola iff relates to supporting line}.}

\end{figure}

\end{proof}
Lambert's Theorem on the parabola states that the circumcircle of
a triangle formed by three tangents to the parabola always passes
through the focus. Using Theorem \ref{thm:tangent to parabola iff relates to supporting line},
we can prove the statement quite easily. Let three trangents $l_{1},l_{2},l_{3}$
to the parabola be given. Then the orthogonal projections of $F$
into $l_{1},l_{2},l_{3}$ all lie on $\Lambda$, and are therefore
collinear. By the Simson-Wallace Theorem, $F$ lies on the circumcircle
of the triangle formed from $l_{1},l_{2},l_{3}$. We introduce a converse
to Lambert's Theorem.
\begin{thm}
\label{thm:Lambert Converse}(Converse to Lambert's Theorem) Let $l_{1}$
and $l_{2}$ be two distinct lines tangent to a parabola $G$ with
focus $F$. Let $I=l_{1}\cap l_{2}$ be their intersection and consider
any circle $C$ passing through points $F$ and $I$. Let $H_{i}\in C\cap l_{i}$,
for $i=1,2$ with at least one $H_{i}\neq I$. Then the line $H_{1}H_{2}$
is tangent to $G$.\end{thm}
\begin{proof}
If $H_{i}=I$
for some $i$, then the statement clearly holds. So assume that $H_{i}\neq I$
for each $i$. By Theorem \ref{thm:tangent to parabola iff relates to supporting line},
the orthogonal projections of $F$ into $l_{1}$ and $l_{2}$ lie
on $\Lambda$. Since $F$ is on the circumcircle of $\triangle H_{1}H_{2}I$,
its pedal is a line (by Theorem \ref{thm:(Simson-Wallace-Theorem)-Given}).
As a line is uniquely determined by two points, this line must be
$\Lambda$. Applying Theorem \ref{thm:tangent to parabola iff relates to supporting line}
again yields that $H_{1}H_{2}$ is tangent to $G$.

\end{proof}

\begin{figure}[H]
\noindent \begin{centering}
\includegraphics[scale=0.45]{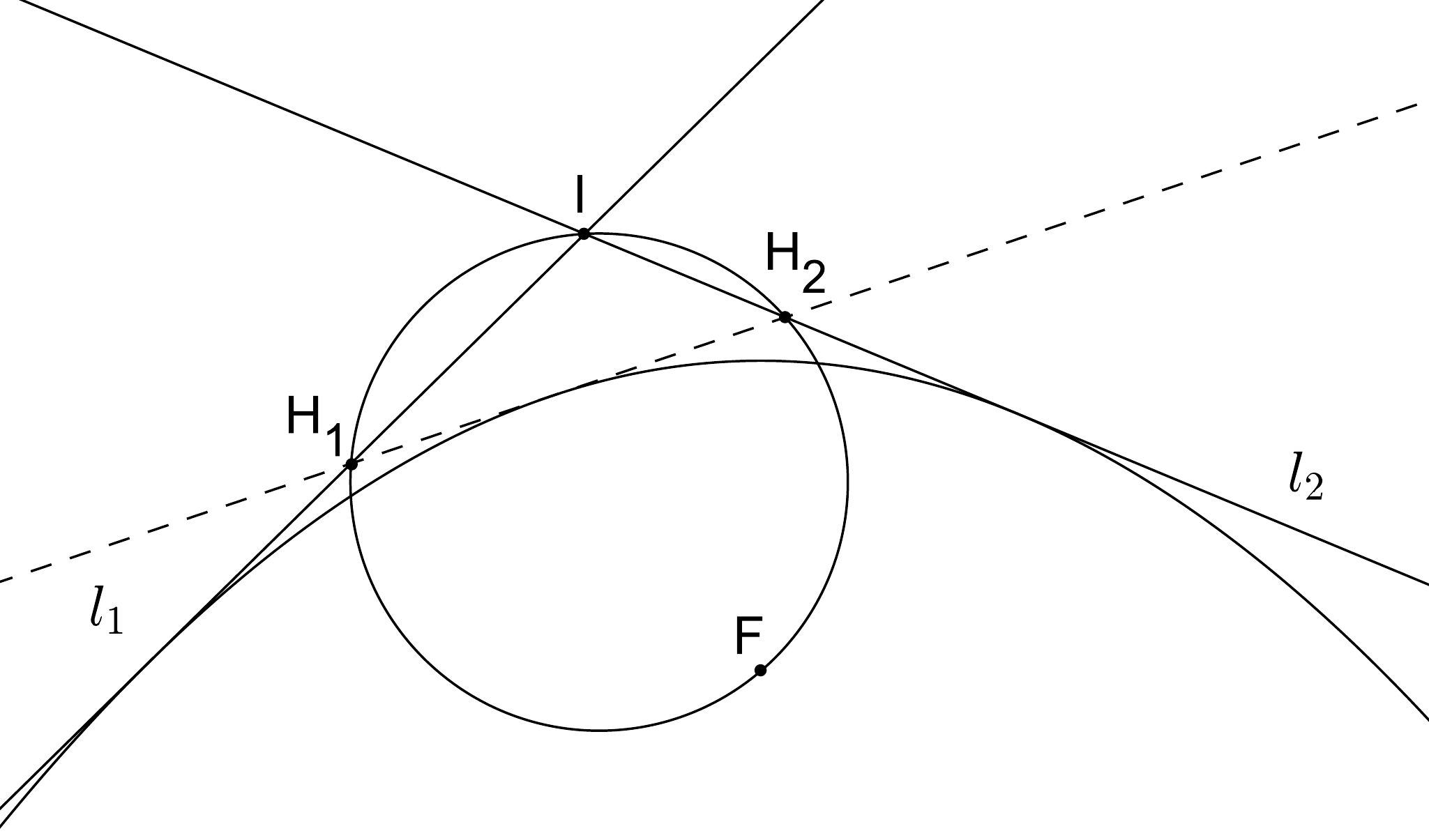}
\par\end{centering}

\caption{Illustration of Theorem \ref{thm:Lambert Converse}.}
\end{figure}

\section{Parbelos}

Recall that the latus rectum of a conic is the chord through the focus
parallel to the conic's directrix. The parbelos is constructed as
follows. Given three points $C_{1},C_{2},C_{3}$ on a line, construct
parabolas $G_{1},G_{2},G_{3}$ that open in the same direction and
whose latera recta are $C_{1}C_{2}$, $C_{2}C_{3}$ and $C_{1}C_{3}$,
respectively. The parbelos is defined as the region bounded by the three latus rectum arcs.

\begin{figure}[H]
\noindent \begin{centering}
\includegraphics[scale=0.40]{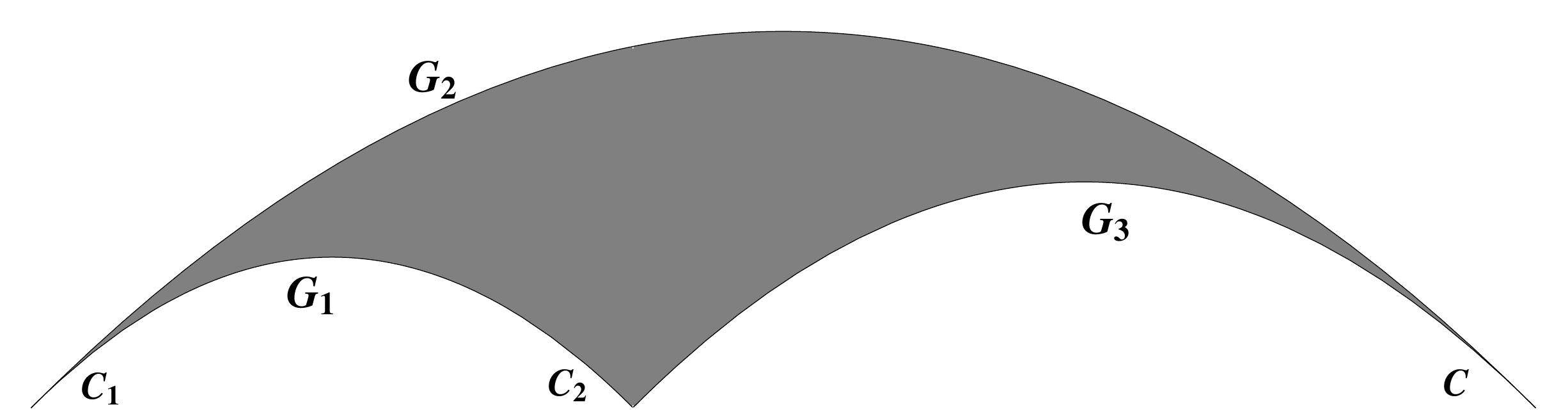}
\par\end{centering}

\end{figure}

The tangent line of a parabola at either endpoint of its latus rectum
forms an angle of $\frac{\pi}{4}$ with the latus rectum. As such,
parabolas $G_{1}$ and $G_{2}$ share the same tangent at $C_{1}$,
and similarly parabolas $G_{2}$ and $G_{3}$ share a tangent at $C_{3}$.
At cusp $C_{2}$, however, we obtain two different tangent directions.
One can extend these four tangents to form a rectangle whose vertices
are the intersections of tangent lines as in Figure \ref{fig:Sondow's-Tangency-Property:}.
We will denote the vertices of this rectangle by $C_{2},T_{1},T_{2},T_{3}$. 

In his paper \cite{Parbelos}, Sondow asks for a synthetic proof of
the following Theorem, which he proves via analytic Geometry:
\begin{thm}
(Sondow's Tangency Property) In the tangent rectangle of the parbelos,
the diagonal opposite the cusp is tangent to the upper parabola. The
contact point lies on the bisector of the angle at the cusp.\end{thm}
\begin{proof}
Let us inscribe the tangent rectangle $r=T_{1}T_{2}T_{3}C_{2}$ in
another rectangle $R$ whose sides are parallel and orthogonal to
$C_{1}C_{3}$. At the cusp $C_{2}$, the angles formed between $C_{1}C_{3}$
and $C_{2}T_{1}$, $C_{2}T_{3}$ are $\frac{\pi}{4}$ and $\frac{3\pi}{4}$.
Using right triangles, it is easy to see that $R$ must be a square
and its center $O$ is the same as that of $r$. 

\begin{figure}[H]
\begin{centering}
\includegraphics[scale=0.3]{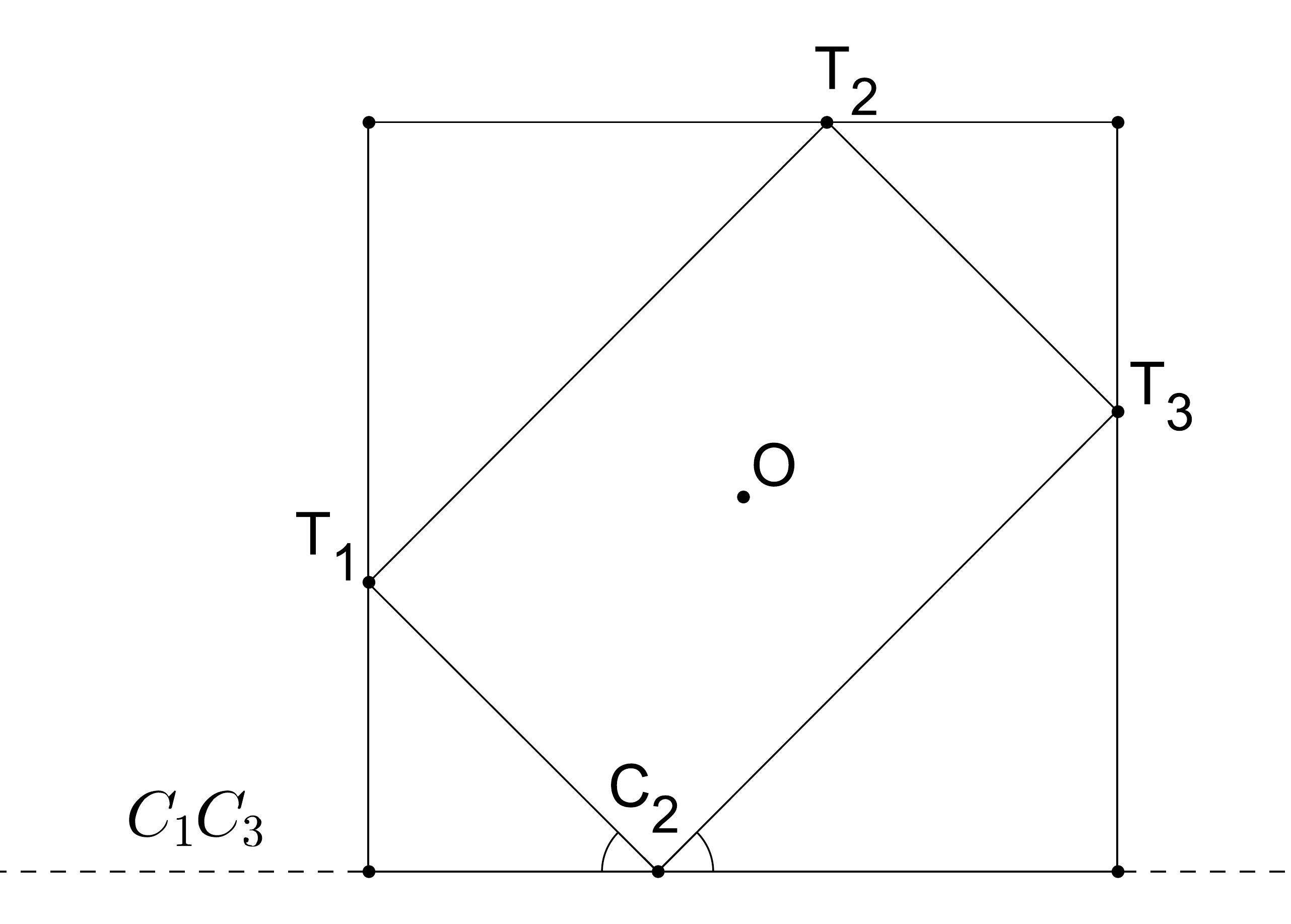}
\par\end{centering}

\caption{\label{fig:billiard table}Rectangle $R$ circumscribing the tangent
rectangle $r$.}

\end{figure}

Consider the circumscribing circle of $r$.

\begin{figure}[H]
\begin{centering}
\includegraphics[scale=0.3]{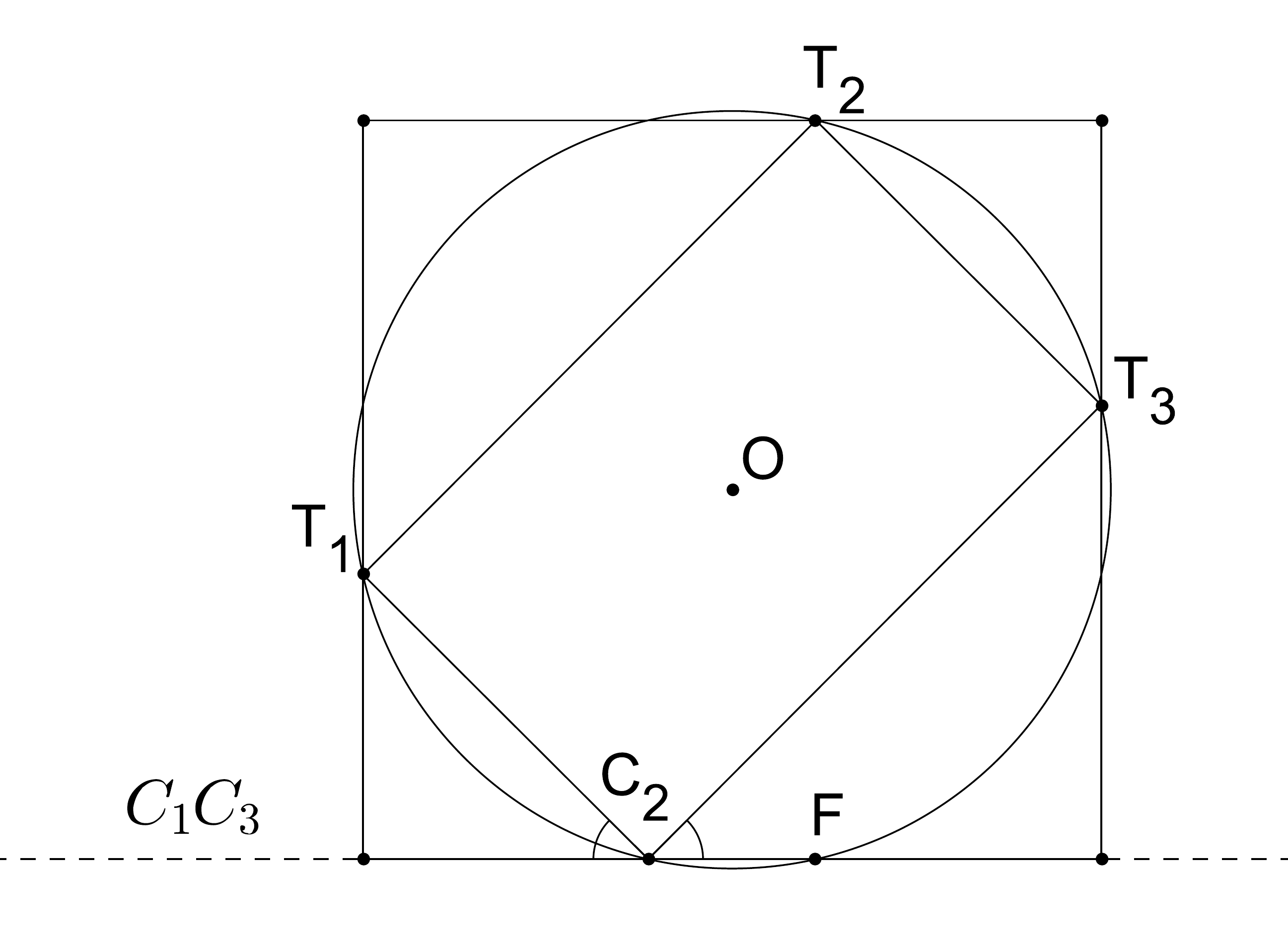}
\par\end{centering}

\caption{The angles at the cusp $C_{2}$ are equal.}
\end{figure}

Since its center is $O$, by symmetry it intersects $C_{1}C_{3}$
at a point $F$, such that $F$ is the orthogonal projection of $T_{2}$
down to $C_{1}C_{3}$. This point is the focus of the outer parabola.
Since $F,C_{2},T_{1},T_{2},T_{3}$ lie on a circle, Theorem \ref{thm:Lambert Converse}
implies that $T_{1}T_{3}$ is tangent to the parabola. 

As for the angle bisector at the cusp $C_{2}$, the billiard angle
property implies that it is orthogonal to $C_{1}C_{3}$ (see Figure
\ref{fig:angle bisector in billiard square}). Let $H$ be the intersection
of the angle bisector with the top of the rectangle (i.e., with the
directrix of the outer parabola).
We would like to see that $FT=HT$. 

\begin{figure}[H]
\begin{centering}
\includegraphics[scale=0.28]{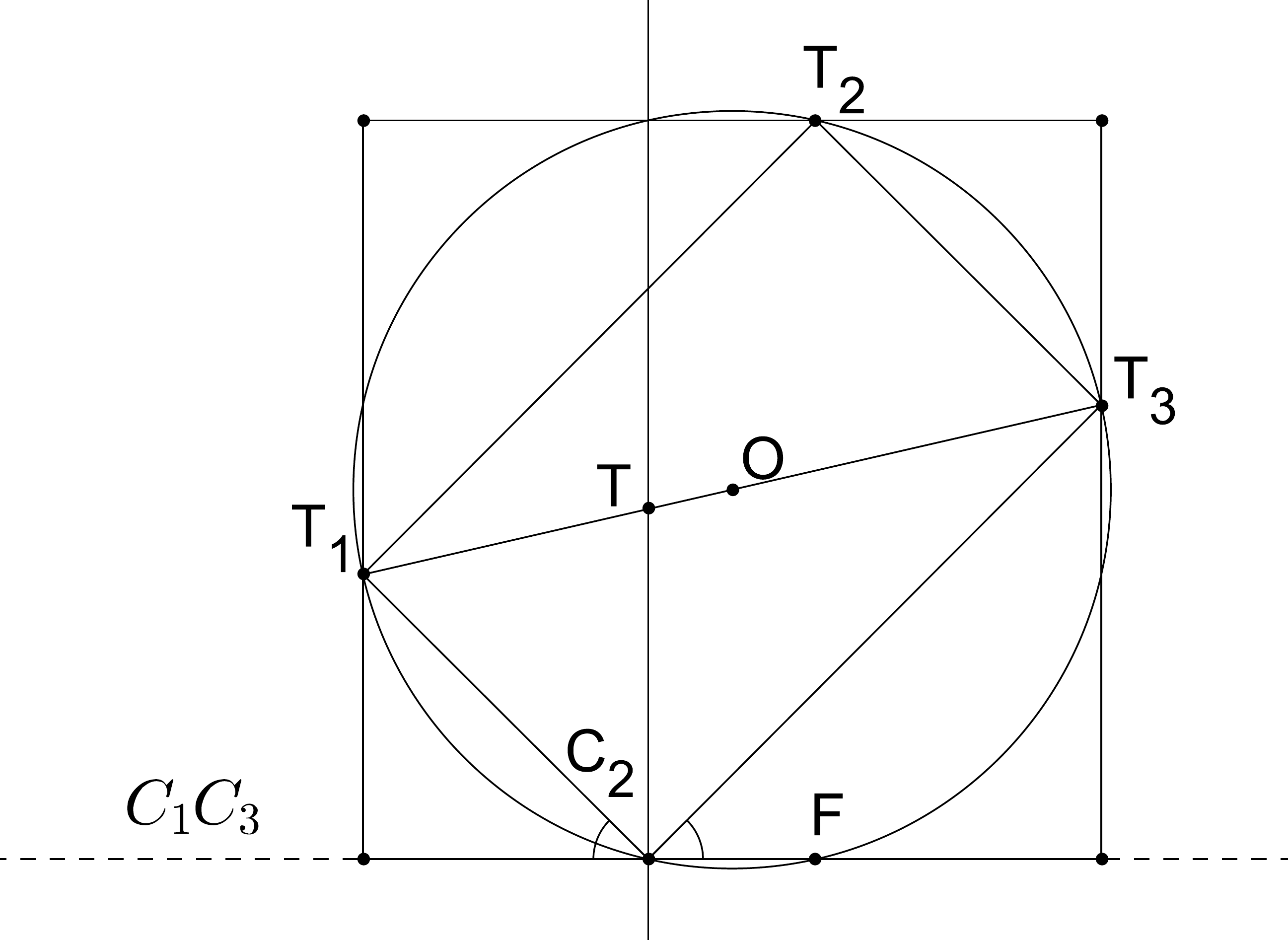}
\includegraphics[scale=0.28]{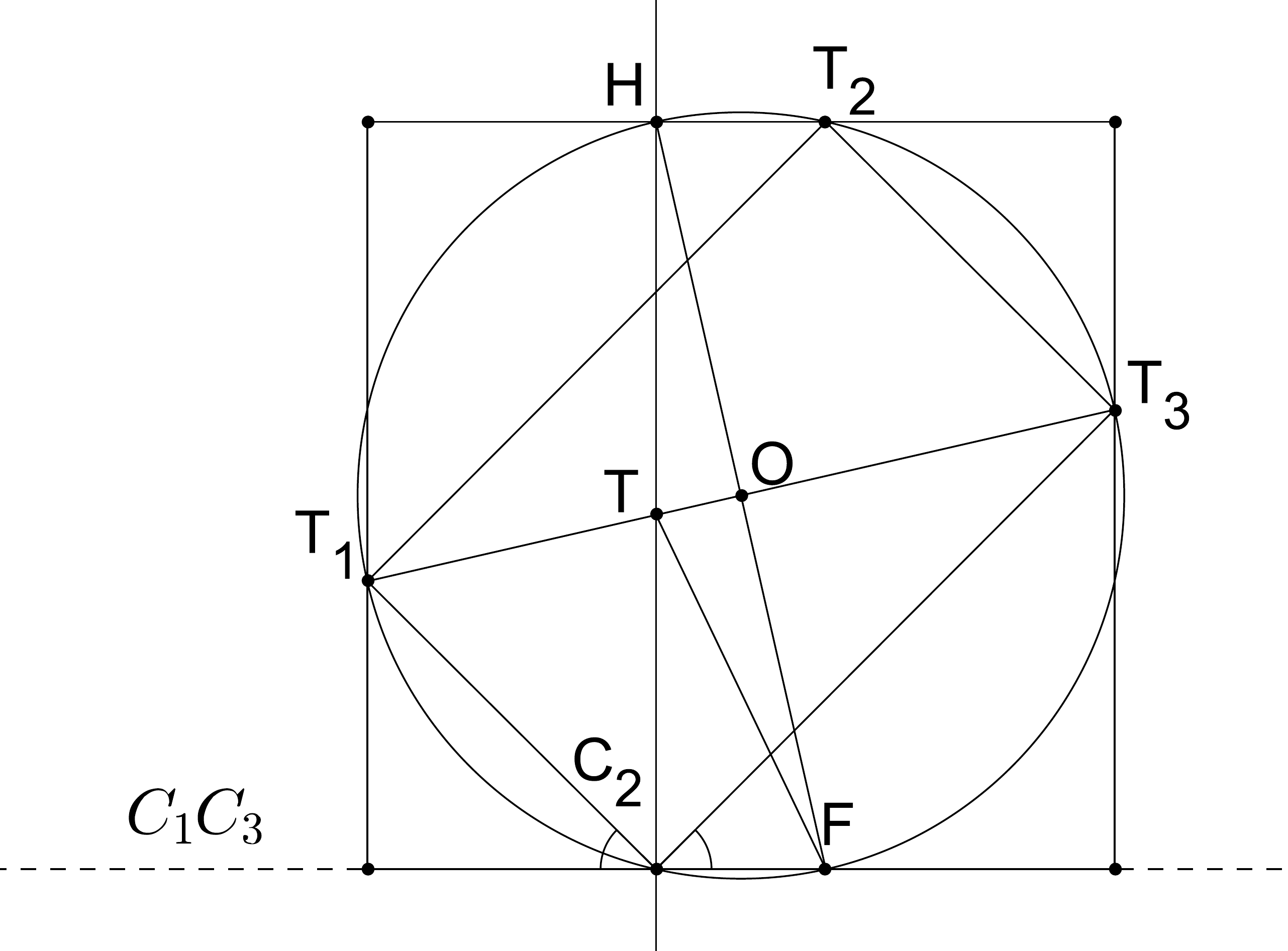}
\par\end{centering}

\caption{\label{fig:angle bisector in billiard square}The angle bisector at
cusp $C_{2}$ is orthogonal to $C_{1}C_{3}$.}
\end{figure}

This is not too hard to see from the Figure \ref{fig:angle bisector in billiard square} on the right. It also follows from the diagram that $F$ is equidistant from $T_{1}$
and $T_{3}$. \end{proof}
\begin{rem}
One way to look at the configuration in Figure \ref{fig:billiard table}
is as a $4$-periodic billiard trajectory in a square billiard table
(for an exposition on billiards in polygons, see \cite{Tabachnikov}).
It would be interesting to see whether there is a deeper connection
between (p)arbelos and billiards. 
\end{rem}
Let $A_{1}$ be the intersection of the axis of symmetry of $G_{1}$
with the directrix of $G_{3}$ (i.e., the line parallel to $C_{1}C_{3}$
and passing through $T_{3}$). Define $A_{3}$ similarly. Notice that the vertical sides of the rectangle $R$ are the axes of symmetry of $G_1$ and $G_3$, while the horizontal sides are $C_1 C_3$ and the directrix of the outer parabola. As a consequence, we obtain the following new
properties for the parbelos.
\begin{cor}
1. The focus $F$ of the outer parabola is equidistant from vertices
$T_{1}$ and $T_{3}$ of the tangent rectangle.

2. The intersection $H$ of the angle bisector at cusp $C_{2}$ and
the directrix of the outer parabola lies on the circumcircle $F,C_{2},T_{1},T_{2},T_{3}$
of the tangent rectangle.

3. This point $H$ is equidistant from vertices $T_{1}$ and $T_{3}$.

4. Points $A_{1}$ and $A_{3}$ lie on circle $F,C_{2},T_{1},T_{2},T_{3}$.

5. Point $A_{1}$ is equidistant from $C_{2}$ and $T_{2}$ and so
is point $A_{3}$. 
\end{cor}
\begin{figure}[H]
\begin{centering}
\includegraphics[scale=0.4]{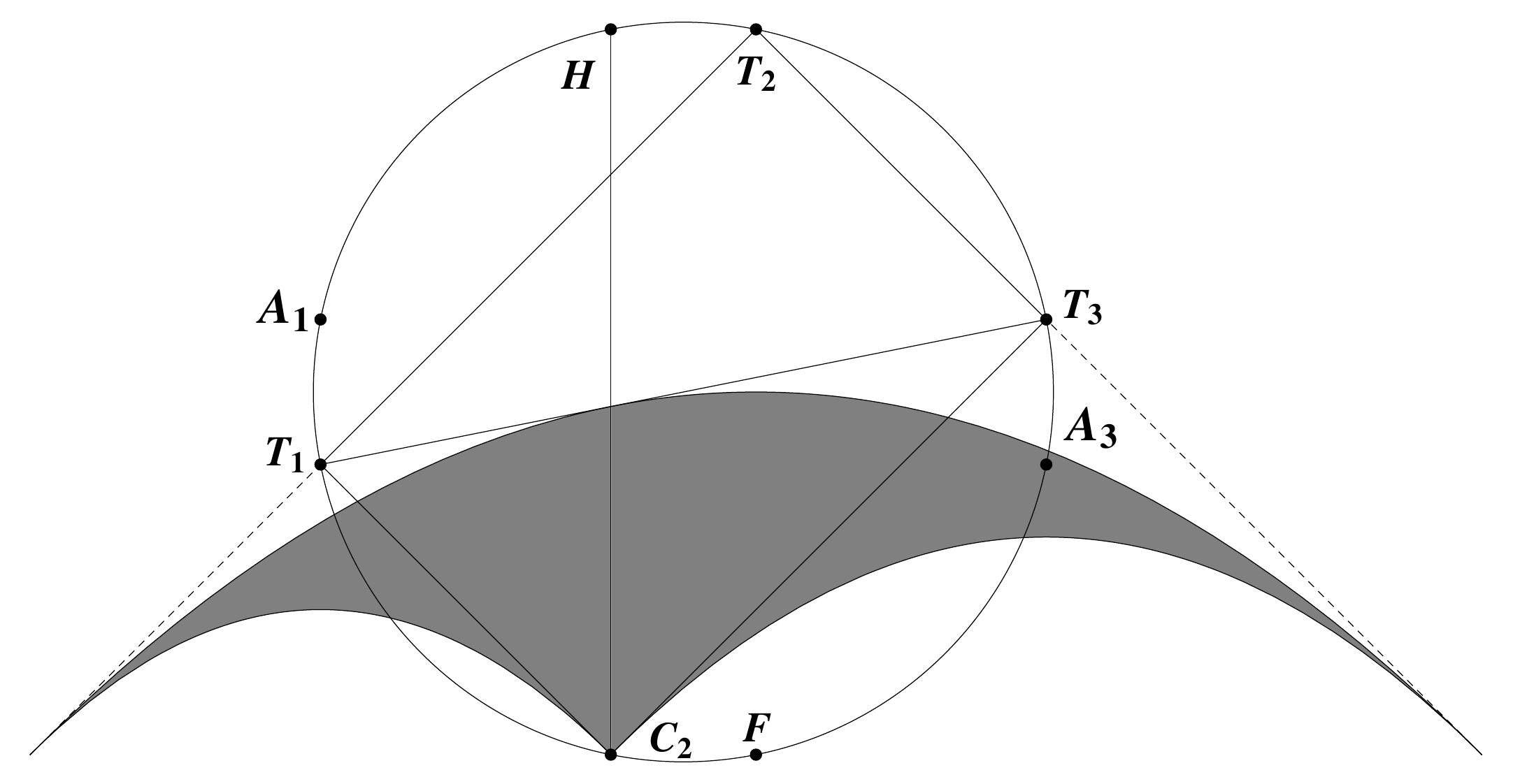}
\par\end{centering}

\caption{The circumcircle of the tangent rectangle and notable points lying
on it.}

\end{figure}

\begin{onehalfspace}
$ $

Emmanuel Tsukerman: Stanford University

\textit{E-mail address: emantsuk@stanford.edu}\end{onehalfspace}

\end{document}